\documentclass[12pt,a4paper]{article}
\usepackage{amssymb}
\usepackage[utf8]{inputenc}
\usepackage[pdftex]{hyperref}
\usepackage[english]{babel}
\usepackage{amssymb,latexsym,amsmath,amscd}
\usepackage[thmmarks, amsmath, amsthm]{ntheorem}
\usepackage{makeidx}
\usepackage{graphicx}
\usepackage{xypic}
\usepackage{titlefoot}

\newtheorem{Def}{Definition}
\newtheorem{Thm}{Theorem}
\newtheorem{Prop}{Proposition}
\newtheorem{Lem}{Lemma}
\newtheorem{Cor}{Corollary}
\newtheorem{Ex}{Example}

\newtheorem{Rem}{Remark}
\theoremstyle{plain}
\theoremnumbering{Alph}
\newtheorem{introtheorem}{Theorem}

\DeclareMathOperator{\uE}{\underline{E}}

\newcommand{\N }{\mathbb N}
\newcommand{\Z }{\mathbb Z}
\newcommand{\R }{\mathbb R}

\newcommand{\C}{\mathbb{C}}
\newcommand{\ra }{\rightarrow}

\title{\emph{K}-homology and \emph{K}-theory for the lamplighter groups of finite groups}
\author{Ram\'{o}n FLORES, Sanaz POOYA and Alain VALETTE}

\begin{document}

\maketitle

\baselineskip=16pt

\unmarkedfntext{Keywords : \emph{K}-theory, $C^*$-algebra, proper actions, Baum-Connes conjecture. MSC classification: 46L80, 55R40}
\unmarkedfntext{The first author was partially supported by the project MTM2010-20692, Ministerio de Ciencia e Innovaci\'on, Spain. The paper was completed as the third author was in a residence at MSRI (Berkeley, CA) as a Simons research professor during the Fall 2016 semester; he acknowledges support from grant DMS-1440140 of the National Science Foundation.}

\begin{abstract} Let $F$ be a finite group. We consider the lamplighter group $L=F\wr\Z$ over $F$. We prove that $L$ has a classifying space for proper actions $\uE L$ which is a complex of dimension two. We use this to give an explicit proof of the Baum-Connes conjecture (without coefficients), that states that the assembly map $\mu_i^L:K_i^L(\uE L)\rightarrow K_i(C^*L)\;(i=0,1)$ is an isomorphism. Actually, $K_0(C^*L)$ is free abelian of countable rank, with an explicit basis consisting of projections in $C^*L$, while $K_1(C^*L)$ is infinite cyclic, generated by the unitary of $C^*L$ implementing the shift. 
Finally we show that, for $F$ abelian, the $C^*$-algebra $C^*L$ is completely characterized by $|F|$ up to isomorphism.
\end{abstract}

\section{Introduction}
The Baum-Connes conjecture for a discrete group $G$ states that the assembly map
$$ \mu_i^G : K_i^G(\uE G)\ra K_i(C_{\mathrm r}^*(G)),\,\,\,\,\, i=0, 1$$ is an isomorphism between two abelian groups: on the left hand side the  $G$-equivariant \emph{K}-homology of $\uE G$, that will be referred as the topological side; and on the right hand side the \emph{K}-theory of $C_{\mathrm r}^*(G)$, to which we refer as the analytical side (see \cite{BCH94}).

Although this conjecture is still open in some cases (including the paradigmatic instance $SL(3, \Z)$), it is proved for a large class of groups including a-T-menable groups (in particular amenable groups), by prominent work of Higson-Kasparov in \cite{HK97}. Naturally it is of strong interest to add more example of groups to this list, but this is not the goal of this article. In contrast, we focus on one specific class of amenable groups and compute their \emph{K}-theory directly, avoiding in particular the use of the Dirac-dual Dirac method or even \emph{KK}-theory at all.\\
Our motivation is to understand better the \emph{K}-groups of these groups and the way they are identified by the assembly map.
Moreover, we go beyond the fact that $\mu_i^G:K_i^G(\uE G)\rightarrow K_i(C^*(G))$ is abstractly an isomorphism (as provided by the Higson-Kasparov approach) and present explicit sets of generators which are identified by the assembly map.

In this article we study the Baum-Connes conjecture for the class of lamplighter groups $L= F \wr \Z =( \oplus_{\Z}F) \rtimes_{\alpha} \Z$, with $F$ a (non-trivial) finite group and $\alpha$ the shift.
For the case of classical lamplighter groups  $L = (\Z /{{n\Z}}) \wr \Z$ computations for the \emph{K}-theory of  $C^*L$  have implicitly appeared in different contexts, see for instance (\cite{AE15}, Example 6.10), (\cite{Bel92}, Theorem 15), (\cite{CM04}, Theorem 4.12), (\cite{CEL13}, Section 6.5) or (\cite{Ohh15}, Example 3). Our approach is different from the previous ones in two ways: first, we treat lamplighter groups of finite groups in full generality; and second, we specify generators for the \emph{K}-groups and show their relevance in the context of Baum-Connes.

Let $ B=\oplus_{\Z}F$, denote by
Min F the set of equivalence classes of minimal projections in the group ring $\C F$, and denote by $\hat F $ the set of equivalence classes of irreducible unitary representations of $F$. We summarise now the result of our calculations in Proposition \ref{topolside} and Theorem \ref{K-theory for $C^*L$}:
\begin{itemize}
	
	\item $K_0^L(\uE L)$ is a free abelian group of countable rank, and its basis is indexed by the orbits of the shift $\Z \curvearrowright \hat F^{(\Z)}$.
	\item $K_0(C^*L)$ is a free abelian group of countable rank, and its basis is indexed by the orbits of the shift $\Z \curvearrowright {Min F^{(\Z)}}$.
	\item $K_1^L(\uE L)$ is an infinite cyclic group; actually the inclusion $\Z\rightarrow L$ induces an isomorphism $\Z=K_1^\Z(\uE \Z)\rightarrow K_1^L(\uE L)$.
	\item $K_1(C^*L)$ is an infinite cyclic group generated by the $K_1$-class of the unitary $u$ which generates the shift. 	
\end{itemize}
Putting all these facts together, we deduce the Baum-Connes isomorphism as a result of our \emph{K}-theoretic computations (Theorem \ref{BCforL}):

\begin{introtheorem}\label{BClamp}
	Let $L=F\wr\Z$ be the lamplighter group over some (non-trivial) finite group $F$. The map $\mu_0^L:K_0^L(\uE L)\ra K_0(C^*L)$ is an isomorphism between two free abelian groups of countably infinite rank. The map $\mu_1^L:K_1^L(\uE L)\ra K_1(C^*L)$ is an isomorphism between two infinite cyclic groups.
\end{introtheorem}

On page 21 of their seminal paper \cite{BC00}, P. Baum and A. Connes formulated the {\it trace conjecture}. It says that for a group $G$ the image of $\tau _* : K_0(C_{\mathrm r}^*(G))\rightarrow \R$ is the {\it subgroup} of $\mathbb{Q}$ generated by all numbers $\frac{1}{|H|}$, where $H\subset G$ runs though all finite subgroups of G. After counterexamples to the trace conjecture were found (see \cite{Roy99}), W. L\"uck formulated in \cite{Luc02} the {\it modified trace conjecture}: the image of $\tau _*$ should be contained in the {\it subring} of $\mathbb{Q}$ generated by the same $\frac{1}{|H|}$'s; he proved (Theorem 0.3 in \cite{Luc02}) that this follows from surjectivity of the the assembly map. For our lamplighter groups $L=F\wr\Z$, this gives immediately $\operatorname{Im} \tau_* = \Z [\frac{1}{|F|}]$. In Proposition \ref{trace} we give a direct proof of this fact, based on our K-theory computations.


It follows from Theorem \ref{BClamp} that the \emph{K}-theory of $C^*L$ does {\it not} depend on the finite group $F$; i.e., varying the group $F$, the corresponding $C^*$-algebras $C^*L$ cannot be distinguished by $K$-theory. It seems a natural question to classify the $C^*L$'s up to $*$-isomorphism. We have only partial results, which however yield a complete classification when $F$ is abelian.

\begin{introtheorem}\label{classify} Let $F_1,F_2$ be finite groups, with $F_1$ abelian. The following are equivalent:
\begin{enumerate}
\item $C^*(F_1\wr\Z)\simeq C^*(F_2\wr\Z)$;
\item $F_2$ is abelian and $|F_1|=|F_2|$.
\end{enumerate}
\end{introtheorem}

Our paper is organized as follows.

Sections 2 and 3 are preliminaries. Section 2 is about Bredon homology, which is useful in computing $K^G_*(\uE G)$. In particular, since the semi-direct product decomposition $L = B \rtimes_{\alpha} \Z$ of the lamplighter groups is the key to the computations of this article, we explain the Mart\'{i}nez spectral sequence \cite{Mar02} that enables to compute Bredon homology of extensions in a similar way Lyndon-Hochschild-Serre sequence computes ordinary homology; Section 3 is about invariants and co-invariants of group actions on free abelian groups, that turn out to be useful in subsequent computations. Section 4 provides a verification of the Baum-Connes conjecture for the locally finite group $B$. The conjecture is inherited by co-limits of groups, so it holds for abstract reasons for $B$, but for later use we need explicit computations, based on the existence of a 1-dimensional model (i.e. a tree) for $\uE B$.

Section 5 is the core of the paper. Building on a construction by Fluch \cite{Flu11}, we obtain an explicit two-dimensional model for $\uE L$. Because of dimension 2, the Mart\'{i}nez spectral sequence degenerates on the $E_2$-page and allows for a detailed computation of the topological side of Baum-Connes. On the analytical side,  the crossed product decomposition $C^*L=C^* B \rtimes_{\alpha} \Z$ allows to appeal to the Pimsner-Voiculescu 6-term exact sequence \cite{PV80} as the main tool for computations.  Having explicit sets of generators on both sides we can prove Theorem \ref{BClamp} and check the trace conjecture.




In section 6, we present some applications of our results to topological dynamics. If we allow the finite group $F$ in $L = F \wr \Z$ to be abelian then we get a well-known dynamical system $(X, \Z , \alpha)$, with $X = \hat B = \prod \hat F $ a Cantor set and  $\alpha $ the full shift.  We present a very explicit basis for $K_0(C(X)\rtimes_{\alpha} \Z)$ (which was known to be a free abelian group of countable rank, see \cite{PT82}). In addition to this, we give a simple algebraic proof for a known theorem of Liv\u{s}ic \cite{Liv72} characterizing $\alpha$-coboundaries in $C(X,\Z)$.

Finally in section 7 we prove Theorem \ref{classify} and show that, for every finite group $F$, the $C^*$-algebra $C^*L$ has stable rank 2.

\medskip
{\bf Acknowledgements}: We thank I. Chatterji, D. Kerr, W. L\"uck, N.C. Phillips and D. Zagier for some useful conversations or exchanges, and F. Duchene for her help with the pictures.

\section{Background on the left hand side of Baum-Connes}

\subsection{Classifying space for proper actions}

Here we discuss classifying spaces for families, which is the main object of study of the left-hand side of Baum-Connes conjecture. We provide a mild introduction, the interested reader is referred to the excellent survey \cite{Luc05}.



Let $G$ be a discrete group. A topological space $X$ is a {\it $G$-space} if $G$ acts by homeomorphisms on $X$.

\begin{Def}

Let $G$ be a group, $H\leq G$ a subgroup, $X$ an $H$-space. Consider the action of $H$ over the product $G\times X$ given by $h(g,x)=(gh^{-1},hx)$. The \emph{induced $G$-space} $G\times_H X$ is  defined as the orbit space $(G\times X)/H$ under the previous $H$-action, endowed in turn with the action of $G$ given by $g[g',x]=[gg',x]$. It is easy to see that this action is well-defined.

\end{Def}

Now we are ready to define the concept of proper action:

\begin{Def}

An action of a discrete group $G$ on a $G$-space $X$ is called \emph
{proper} is there are finite subgroups $G_i<G$ and open subspaces $X_i\subset X$ with $X_i$ invariant under the action of $G_i$, such that the $G$-maps $G\times_{G_i}X_i\hookrightarrow X$ are embeddings, and $X=\bigcup_iG\times_{G_i}X_i$.
\end{Def}

It is easy to see that the isotropy groups of a proper $G$-action are finite subgroups of $G$.

\begin{Def}

A \emph{$G$-CW-complex} is a $G$-space endowed with a filtration by closed subspaces $X^0\subset X^1\subset \ldots \subset X$, with $X_0$ discrete, $X=\bigcup X^i$ for every $i$, and such that:

\begin{itemize}

\item A subspace of $X$ is closed if and only if its intersection with every element of the filtration is so.

\item Every $X^i$ is build from $X^{i-1}$ using a push-out

$$\xymatrix{S^{n-1}\times \Delta_n \ar[r] \ar[d] & X^{n-1} \ar[d]\\
B^{n}\times \Delta_n \ar[r] & X^n}$$ with $\Delta_i$ a discrete $G$-space for every $i$. A $G$-CW-complex is said $n$-dimensional if $X^n=X$ for a certain $n$.

\end{itemize}

\end{Def}

It is known that a $G$-CW complex is proper if and only if all the isotropy groups are finite.

Let now $\mathfrak{F}$ be a family of subgroups of a group $G$ that is closed under passing to subgroups and conjugation.

\begin{Def}

A \emph{classifying space for the family $\mathfrak{F}$} is a $G$-CW-complex $E_{\mathfrak{F}}G$ such that $X^H$ is contractible for every $H\in\mathfrak{F}$ and empty if $H\notin\mathfrak{F}$. If $\mathfrak{F}$ is the family of finite subgroups of $\mathfrak{F}$, $E_{\mathfrak{F}}G$ is usually denoted $\underline{E}G$, and called the \emph{classifying space for proper actions of $G$}.
\end{Def}

There is a universal construction that always guarantees the existence of a model for $E\mathfrak{F}$, and this classifying space is always unique up to $G$-homotopy equivalence (\cite{BCH94}, 1.6-1.7.)

\subsection{Bredon homology}

\subsubsection{Generalities}

In this section we introduce the main concepts concerning Bredon homology. More details can be found, for example, in \cite{MV03}, part I.

As it happens in the ordinary case, Bredon homology of a group is defined in terms of projective resolutions, that can be constructed algebraically or by using a cellular chain complex. We will recall briefly the necessary definitions for the second approach, while the reader interested in a more algebraic treatment is referred to \cite{MV03}, pages 8-9.

\begin{Def}

Let $G$ be a discrete group, $\mathfrak{F}$ a collection of subgroups of $G$. The \emph{orbit category} $\mathfrak{O}_{\mathfrak{F}}G$ is the category whose objects are the left cosets $G/H$ with $H\in\mathfrak{F}$, and whose morphisms are the $G$-maps $G/H\rightarrow G/K$. The category of contravariant (resp. covariant) functors from $\mathfrak{O}_{\mathfrak{F}}G$ to abelian groups is denoted by $Mod_{\mathfrak{F}G}$ (resp. $G$-$Mod_{\mathfrak{F}}$), and its objects are called $\mathfrak{O}_{\mathfrak{F}}G$-modules. A $\mathfrak{O}_{\mathfrak{F}}G$-module $P$ is called \emph{projective} is the functor $Mor(P,-)$ from $Mod_{\mathfrak{F}G}$ to abelian groups is exact.

\end{Def}

As the category $\mathfrak{O}_{\mathfrak{F}}G$ is abelian, projective resolutions can be defined, and it can be seen that the category has enough projectives. In order to construct homology groups, we should be able to construct such resolutions in a systematic way. As the nature of our work is topological, we will first construct projectives associated to $G$-spaces, and then we will use them to build projective resolutions of $G$.

Let us define now the cellular complex. Let $X$ be a $G$-CW-complex, and let $\mathfrak{F}$ be a family of subgroups of $G$ that contains the isotropy groups of the action. Let $\Delta_i$, for $i\geq 0$, be the discrete spaces that appear in the definition of $X$ as a $G$-pushout. Then for every $H\in\mathfrak{F}$, we define $C_i(X^H)=\mathbb{Z}[\Delta_i^H]$, the free abelian group with base the fixed-point set $\Delta_i^H$. Now we may define a
 $\mathfrak{O}_{\mathfrak{F}}G$-module $\underline{C_i(X)}$ by assigning to every homogeneous space $G/H$ the abelian group $C_i(X^H)$. There are boundary operators $\delta_i:\underline{C_i(X)}\rightarrow\underline{C_{i-1}(X)}$ and an augmentation $\underline{C_0(X)}\rightarrow\underline{\mathbb{Z}}$ (here $\underline{\mathbb{Z}}$ is the constant module with value $\mathbb{Z}$) induced by the corresponding boundaries and augmentation of the cellular chain complexes of the fixed point spaces. Note that if $i>dim(X)$, $\underline{C_i(X)}$ is trivial.

 \begin{Def}

 Given $\mathfrak{O}_{\mathfrak{F}}G$-modules $M$ and $N$, its \emph{categorical tensor product} $M\bigotimes_{\mathfrak{F}}N$ is the quotient abelian group $\bigoplus_{\mathfrak{F}}M(G/K)\otimes N(G/K)$, by the relation generated by $M(\phi)(m)\otimes n\sim m\otimes N(\phi)(n)$, for $m\in M(G/K)$, $n\in N(G/K)$, and $\phi:G/K\rightarrow G/K$ a morphism. In this way a functor $-\otimes_{\mathfrak{F}} N:Mod_{\mathfrak{F}G}\rightarrow\mathfrak{Ab}$ is defined.

 \end{Def}

Now we are ready to define Bredon homology groups with coefficients for the case of a $G$-space:

\begin{Def}

If $X$ is a $G$-CW-complex and $N\in G$-$Mod_{\mathfrak{F}}$, with $\mathfrak{F}$ containing all the isotropy groups of the $G$-action in $X$, then the \emph{Bredon homology groups} of $X$ with coefficients in $N$ are defined as $H_i^{\mathfrak{F}}(X;N)=H_i(\underline{C_*(X)}\bigotimes_{\mathfrak{F}}N))$. The boundaries are induced by those of $\underline{C_*(X)}$.
\end{Def}

In particular, if $X=E_{\mathfrak{F}}G$, $\underline{C_*(X)}$ is a projective resolution of $\underline{\mathbb{Z}}$, and this implies that $H_*^{\mathfrak{F}}(E_{\mathfrak{F}}G;N)$ is equivalent to $H_i^{\mathfrak{F}}(G;N)$. Observe that $H_0^{\mathfrak{F}}(G;N)$ can be identified with $colim_{\mathfrak{O}_{\mathfrak{F}}G}N(G/K)$.

\subsubsection{The class of finite subgroups}

In this subsection we assume that $\mathfrak{F}$ is the collection of finite subgroups of $G$. First we define the $\mathfrak{O}_{\mathfrak{F}}G$-module $R_\C\in G$-$Mod_{\mathfrak{F}}$, which are the appropriate coefficients in the context of Baum-Connes conjecture (see Theorem \ref{Mislin2} below). The value of this functor on $G/K$ ($K$ a finite subgroup of $G$) is $R_\C(K)$, the complex representation ring of $K$; and to a $G$-map $G/H\rightarrow G/K$, given by an inclusion $gHg^{-1}\subset K$, we associate the homomorphism
$$Ind_{gHg^{-1}}^K: R_\C(H)=R_\C(gHg^{-1})\rightarrow R_\C(K),$$
where \emph{Ind} denotes the induced representation.

The next two results will be a key tool in the computation of the low-dimensional Bredon homology groups.

\begin{Thm}\label{Mislin1}(Theorem I.3.17 in \cite{MV03}) If $\dim \uE G=1$ (i.e. $\uE G$ is a tree), then $H_i^{\mathfrak{F}}(G,R_{\mathbb{C}})=0$ for $i>1$ and there is an exact sequence
$$0\ra H_1^{\mathfrak{F}}(G,R_{\mathbb{C}}) \ra \bigoplus_{[e]}R_{\mathbb{C}}(G_e)\stackrel{f}\rightarrow \bigoplus_{[v]}R_{\mathbb{C}}(G_v)\ra H_0^{\mathfrak{F}}(G,R_{\mathbb{C}}) \ra 0,$$
where the direct sums are respectively taken over the $G$-orbits of edges and vertices of the tree $\uE G$, and $G_e$ and $G_v$ stand for the isotropy groups of the edges and the vertices, respectively. For $\pi\in R_\C(G_e)$, we have $f(\pi)=(Ind_{G_e}^{G_{e^-}}\pi, -Ind_{G_e}^{G_{e^+}}\pi)\in R_\C(G_{e^-})\oplus R_\C(G_{e^+})$.
\end{Thm}

Aside its intrinsic interest, we would like to compute Bredon homology groups because of their relation with the left hand-side in the Baum-Connes conjecture. In the following, $ K_i^G$ stands for the $G$-homology theory associated to the $\mathfrak{O}(G)$-spectrum $\mathbf{K}^{top,G}$ (more information in \cite{MV03}, I.5). For us, of special interest will be the following theorem of Mislin, that is in turn a simplified version of the Atiyah-Hirzebruch spectral sequence in the case in which there is a low-dimensional model for $\underline{E}G$:

\begin{Thm}{(\cite{MV03}, Theorem I.5.27)}\label{Mislin2}
Let $G$ be an arbitrary group such that there exists a model for $\underline{E}G$ of dimension not bigger than 2. Then there is a natural short exact sequence:
$$0\rightarrow H_0^{\mathfrak{F}}(G;R_{\mathbb{C}})\rightarrow K_0^G(\underline{E}G)\rightarrow H_2^{\mathfrak{F}}(G;R_{\mathbb{C}})\rightarrow 0,$$
and a natural isomorphism $ H_1^{\mathfrak{F}}(G;R_{\mathbb{C}})\simeq  K_1^G(\underline{E}G)$.
\end{Thm}

\subsection{Mart\'{i}nez spectral sequence}
\label{Martinez}
The main tool we are going to use to compute Bredon homology is an appropriate version of the Lyndon-Hochschild-Serre spectral sequence, developed by C. Mart\'{i}nez in \cite{Mar02}, and that we briefly describe in this section.

Let $N\rightarrow G\rightarrow \tilde{G}$ be a group extension, and consider a family $\mathfrak{F}$ of subgroups of $G$. Now we consider another family $\tilde{\mathfrak{h}}$ of subgroups of $\tilde{G}$, such that for any $L\in\mathfrak{F}$ the quotient $LN/N$ belongs to $\tilde{\mathfrak{h}}$; and such that if $N\leq S$ and $S/N$ belongs to $\tilde{\mathfrak{h}}$, then $S\cap K\in\mathfrak{F}$ for every $K\in\mathfrak{F}$. In particular, if these two families are the collections of finite subgroups, the two conditions hold. Consider now $\mathfrak{h}=\{S\leq G:N\leq S\textrm{ and }S/N\in\tilde{\mathfrak{h}}\}$, the ``pullback" of the family $\tilde{\mathfrak{h}}$.

 Let now $G$ be a group and a $\mathfrak{O}_{\mathfrak{F}}G$-module $D$ of coefficients. There is an equivalence of categories $G$-$Mod_{\mathfrak{h}}\rightarrow \tilde{G}$-$Mod_{\tilde{\mathfrak{h}}}$ given by $\tilde{D}(\tilde{G}/\tilde{S})=D(G/S)$, where $S$ is the pullback of $\tilde{S}$ in $G$. In this way it is defined a first quadrant spectral sequence such that $E_{p,q}^2=H_p^{\tilde{\mathfrak{h}}}(\tilde{G},\overline{H_q^{\mathfrak{F}\cap -}(-,D)})$, and converging to $E_{p,q}^{\infty}=H_{p+q}^{\mathfrak{F}}(G,D)$.

Some observations are pertinent here:

\begin{itemize}

\item In the page $E_2$, $\overline{H_q^{\mathfrak{F}\cap -}(-,D)}$ is a module in $G$-$Mod_{\tilde{\mathfrak{h}}}$.

\item The values of $\overline{H_q^{\mathfrak{F}\cap -}(-,D)}$ are computed in the following way: first take an element $\tilde{V}<\tilde{G}$ in $\tilde{\mathfrak{h}}$, and consider its pullback $V$ in $G$. Then consider the family $\mathfrak{F}_V$ of the subgroups in $\mathfrak{F}$ which are subgroups of $V$. Then the value of the functor $\overline{H_q^{\mathfrak{F}\cap -}(-,D)}$ over $\tilde{V}$ is $H_q^{\mathfrak{F}_V}(V,D)$. This fact makes clear the similarity with the classical Lyndon-Hochschild-Serre spectral sequence.

\item There is a corresponding spectral sequence for Bredon \emph{cohomology}, the we will not define here. In fact, the spectral sequence is constructed in a more general framework, see \cite{Mar02}.

\end{itemize}

\section{Invariants and co-invariants}

In our computations for both sides of the Baum-Connes conjecture for lamplighter groups, the relevant \emph{K}-homology and \emph{K}-theory groups will appear as invariants and co-invariants of a certain $\Z$-action on an abelian group. We will spend some time here discussing these concepts.

 For $M$ a $G$-module, we denote by $M^G$ the sub-module of $G$-invariants (i.e. the $G$-fixed points), and by $M_G=M/<m-g.m: m\in M,g\in G>$ the module of $G$-co-invariants.
 If $X$ is a set, $\Z X$ denotes the free abelian group on $X$, i.e. the group of almost everywhere zero functions $X\rightarrow\Z$.

\begin{Lem}\label{inv/coinv} Let $G$ be a countable group, and $X$ be a countable $G$-space; then the space $(\mathbb{Z}X)^G$ of invariants can be identified with $\mathbb{Z}Y$, where $Y$ is the set of finite orbits in $X$; and the space $(\mathbb{Z}X)_G$ of co-invariants can be identified with $\mathbb{Z}(G\backslash X)$. In particular, both are free abelian groups.
\end{Lem}

\begin{proof} Since a $G$-fixed function on $X$ must be constant on $G$-orbits, the first claim is clear. For the second: let $S$ be a system of representatives for $G$-orbits in $X$. Then the map $$I: \mathbb{Z}X\rightarrow \mathbb{Z}(G\backslash X): \phi\mapsto (s\mapsto\sum_{g\in G}\phi(g.s))$$
is onto with kernel precisely $\langle g\phi-\phi:g\in G,\phi\in\mathbb{Z}X\rangle$.
\end{proof}

If $S$ is a system of representatives for the $G$-orbits in $X$, as in the above proof, then $\Z X=<g.m-m:m\in M,g\in G> \oplus \Z S$ gives a splitting of the quotient map $\Z X \rightarrow (\Z X)_G$.

Let $F$ be a finite pointed set, pointed by some element $o\in F$. We denote by $F^{(\Z)}$ the countable set of maps $\Z\rightarrow F$ that are almost everywhere equal to $o$. We want to apply Lemma \ref{inv/coinv} to $X=F^{(\Z)}$, with $G=\Z$ acting by the shift $\alpha$. We describe a set of representatives for the orbits of $\alpha$ on $X$. For $g\in F$, by abuse of notation we also denote by $g$ the element of $F^{(\Z)}$ taking the value $g$ at $k=0$ and $o$ at $k\neq 0$. For  $n\geq 0,\varepsilon\in F^n$, and $g,h\in F\backslash\{o\}$, let $g\varepsilon h$ be the element of $F^{(\Z)}$ defined by
$$(g\varepsilon h)(k)=\left\{\begin{array}{ccc}g & if & k=0 \\\varepsilon_k & if & k=1,...,n \\h & if & k=n+1 \\o &  & otherwise.\end{array}\right.$$

\begin{Lem}\label{repres} A set $S$ of representatives for the $\alpha$-orbits on $F^{(\Z)}$, is:
$$S=\{g:g\in F\}\cup\{g\varepsilon h:n\geq 0,\varepsilon\in F^n,g,h\in F\backslash\{o\}\}.$$
In particular $\Z(F^{(\Z)})$ is the direct sum of $<m-\alpha(m):m\in \Z(F^{(\Z)})>$ and $\Z S$.
\end{Lem}

\begin{proof} It is clear that no two distinct elements of $S$ are in the same orbit for the shift $\alpha$. For $\eta\in F^{(\Z)}$, let the support of $\eta$, denoted by $supp\,\eta$,  be the set of integers $k$ where $\eta_k\neq o$. If $supp\,\eta$ is empty, then $\eta=\delta_0^o$. If $supp\,\eta$ has just one point, then $\eta$ is in the same orbit as some $\delta_0^g$ with $g\neq o$. If $supp\,\eta$ has at least two points, then by shifting the minimum of $supp\,\eta$ to $0$ we get some $g\varepsilon h$.

\end{proof}

\section{The Baum-Connes conjecture for locally finite groups}\label{K-locfin}


Let $F$ be a finite group. Then $C^*F=\C F$, the complex group ring, and $\uE F=\{*\}$, the one-point space. So $K_1(C^*F)=K_1^F(\uE F)=0$ while $K_0(C^*F)$ and $K_0^F(\uE F)$ are both naturally isomorphic to the (complex) representation ring $R_\C(F)$, that is the free abelian group $\Z\hat{F}$ on the set $\hat{F}$ of irreducible representations of $F$. Recall that $\C F=\bigoplus_{\pi\in\hat{F}} M_{\dim\pi}(\C)$, where $M_n(\C)$ denotes the algebra of $n\times n$-matrices with complex coefficients. From Example II.2.11 in $\cite{MV03}$, the assembly map $\mu_0^F:R_\C(F)\ra K_0(\C F)$ maps $\pi\in\hat{F}$ to the class $[e_\pi]$ of a minimal idempotent $e_\pi\in M_{\dim\pi}(\C)$.

When $F$ is a subgroup of the finite group $G$, denote by $i:F\ra G$ the inclusion; we have a commutative diagram
$$\begin{array}{ccccc}R_\C(F) & &\stackrel{Ind_F^G}{\longrightarrow} & &R_\C(G) \\
& & & & \\
\mu_0^F \downarrow & & & & \downarrow\mu_0^G \\
& & & & \\
K_0(C^*F) & & \stackrel{i_*}{\longrightarrow} & &K_0(C^*G), \end{array}$$
which is commutative by Proposition 2.3.1 in \cite{GHJ89}.

Now let $B$ be a (countable) locally finite group, given as the co-limit of an increasing sequence $(B_n)_{n>0}$ of finite groups. By Corollary I.5.2 in \cite{MV03}, $K_i^B(\uE B)$ is the co-limit of the $K_i^{B_n}(\uE B_n)$'s, and by Theorem I.5.10 in \cite{MV03} the assembly map $\mu_i^B$ is the co-limit of the $\mu_i^{B_n}$'s (and the Baum-Connes conjecture holds for $B$). For later computations with lamplighter groups, we need to identify explicitly the co-limits on both sides of the assembly map.

Observe that $B$ acts over $\uE B$, which is a tree basically given by Bass-Serre theory: the sets of vertices and edges are both identified with $\coprod_{n>0} B/B_n$, connecting the edge $gB_n$ the vertices $gB_n$ and $gB_{n+1}$ (for $g\in B$). Hence, $B$ appears as the fundamental group of a ray of groups with vertices $v_1,v_2,...$, vertex group $B_n$ at $v_n$ and edge group $B_n$ at the edge $v_nv_{n+1}$.

\includegraphics[width=12cm,height=12cm]{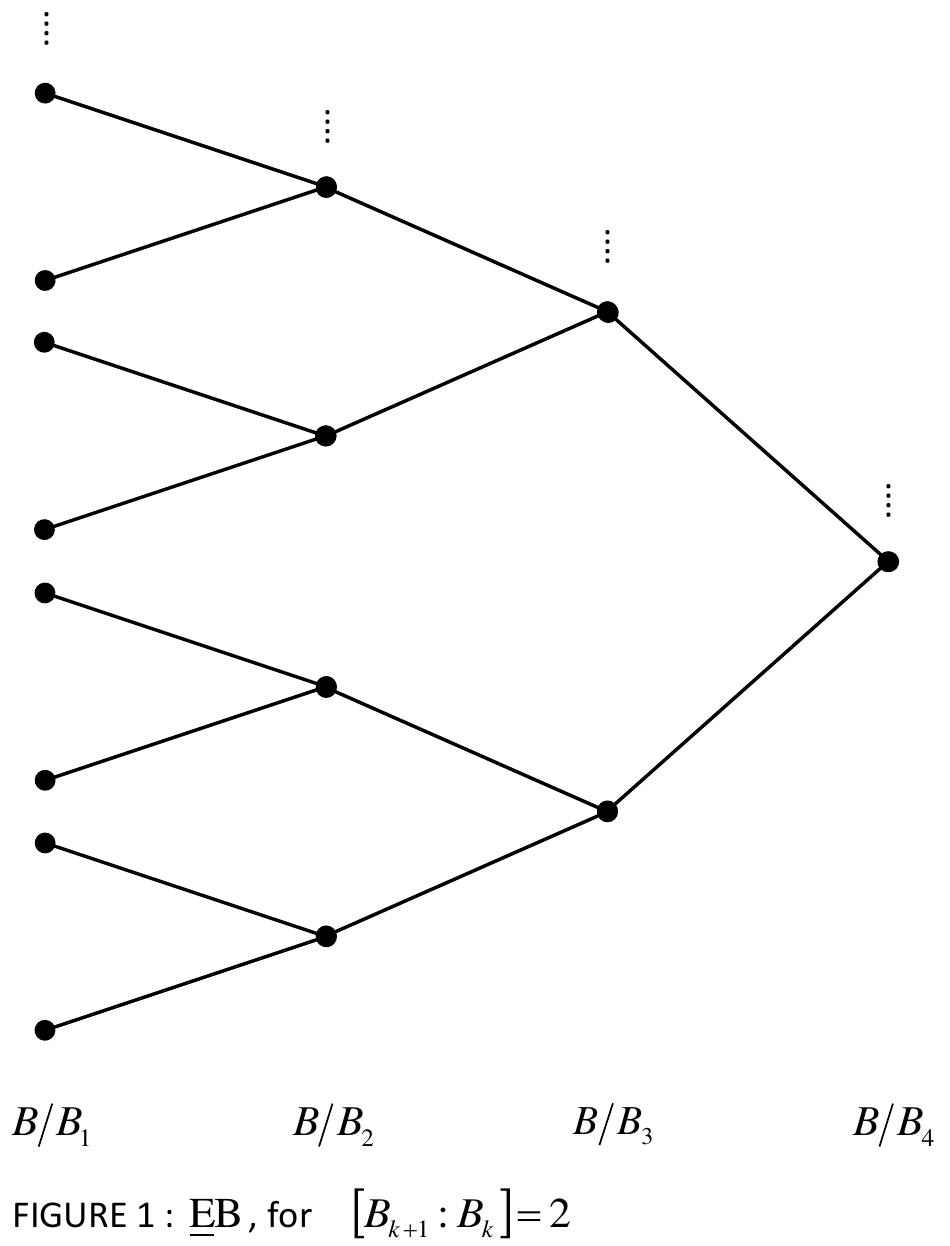}


\begin{Prop}\label{freeabelian1} Let $(F_n)_{n>0}$ be a collection of finite groups, set $B_n=\oplus_{k=1}^n F_k$ and $B=\oplus_{k=1}^\infty F_k$. Embed the dual $\hat{B}_n$ into $\hat{B}_{n+1}=\hat{B}_n\times\hat{F}_{n+1}$ by
$$s_n:\hat{B}_n\ra \hat{B}_{n+1}:\pi_n\mapsto \pi_n\otimes 1_{F_{n+1}},$$
where $1_{F_n}$ denotes the trivial 1-dimensional representation of $F_n$. Then the free abelian group on the co-limit of the system of sets $(\hat{B}_n,s_n)_{n>0}$, is naturally isomorphic to $K_0^B(\uE B)$.
\end{Prop}

\begin{proof} By Theorem \ref{Mislin1}, we have $H_i^{\mathfrak{Fin}}(B,R_{\mathbb{C}})=0$ for $i\geq 1$; so by Theorem \ref{Mislin2}, $H_0^{\mathfrak{Fin}}(B;R_{\mathbb{C}})\simeq K_0^B(\uE B)$. Using the description of $B$ as the fundamental group of a ray of groups, and appealing to Theorem \ref{Mislin1} again, $H_0^{\mathfrak{Fin}}(B;R_{\mathbb{C}})$ is the co-kernel of $f:\bigoplus_{n>0}R_\C(B_n)\ra \bigoplus_{n>0}R_\C(B_n)$. Remark that $f$ maps $\pi_n\in R_\C(B_n)$ to $(\pi_n,-Ind_{B_n}^{B_{n+1}}\pi_n)\in R_\C(B_n)\oplus R_\C(B_{n+1}).$ Now, denoting by $\lambda_{F_n}$ the regular representation of $F_n$, we have by an easy application of Frobenius reciprocity:
$$Ind_{B_n}^{B_{n+1}}\pi_n=\pi_n\otimes \lambda_{F_{n+1}}=\sum_{\sigma\in\hat{F}_{n+1}}\dim\sigma\cdot(\pi_n\otimes\sigma).$$
Since $\hat{B}_{n+1}=\hat{B}_n\times\hat{F}_{n+1}$, we denote by $r_n:\hat{B}_{n+1}\ra\hat{B}_n$ the projection on the first factor, and we may re-write:
$$Ind_{B_n}^{B_{n+1}}\pi_n=\sum_{\pi_{n+1}\in\hat{B}_{n+1}:r_n(\pi_{n+1})=\pi_n}\frac{\dim\pi_{n+1}}{\dim r_n(\pi_{n+1})}\cdot\pi_{n+1}.$$

We now view $\bigoplus_{n>0}R_\C(B_n)$ as $\Z(\coprod_{n>0}\hat{B}_n)$, the group of integer-valued functions with finite support on the disjoint union $\coprod_{n>0}\hat{B}_n$. In that picture we have, for $\phi\in \Z(\coprod_{n>0}\hat{B}_n)$:
\begin{equation}\label{f}
(f(\phi))(\pi_n)=\phi(\pi_n)- \phi(r_{n-1}(\pi_n)).\frac{\dim\pi_n}{\dim r_{n-1}(\pi_n)}
\end{equation}
(where $\pi_n\in\hat{B}_n$).

Consider now the subgroup $H=:\Z(\coprod_{n>0}(\hat{B}_n\backslash s_{n-1}(\hat{B}_{n-1})))$. The proposition will follow from:

{\bf Claim:} $Im\,f\oplus H=\Z(\coprod_{n>0}\hat{B}_n)$ (so $H\simeq Coker\,f$).

To prove the claim, first observe that $Im\,f + H=\Z(\coprod_{n>0}\hat{B}_n)$. To see this, we first prove by induction over $n$ that every $\pi_n\in\hat{B}_n$ (viewed as a basis element of $\Z(\coprod_{n>0}\hat{B}_n)$) is in $Im\, f+H$. This is clear for $n=1$, as $\Z(\hat{B}_1)\subset H$. For $n>1$, this is also clear if $\pi_n\notin s_{n-1}(\hat{B}_{n-1})$ (then $\pi_n\in H$). So assume $\pi_n=s_{n-1}(\pi_{n-1})=\pi_{n-1}\otimes 1_{F_n}$, for some $\pi_{n-1}\in\hat{B}_{n-1}$. Then
$$f(\pi_{n-1})=\pi_{n-1}-\pi_n-\sum_{\sigma\in\hat{F}_n,\sigma\neq 1_{F_n}}\dim\sigma\cdot (\pi_{n-1}\otimes\sigma),$$
or
$$\pi_n=\pi_{n-1}-f(\pi_{n-1})-\sum_{\sigma\in\hat{F}_n,\sigma\neq 1_{F_n}}\dim\sigma\cdot (\pi_{n-1}\otimes\sigma).$$
Using the induction hypothesis for $\pi_{n-1}$, the right hand-side belongs to $Im\,f+H$.

It remains to show that $Im\,f\cap H=\{0\}$, so let $\phi\in \Z(\coprod_{n>0}\hat{B}_n)$ be such that $f(\phi)$ vanishes on $\coprod_{n>0}s_n(\hat{B}_n)$. We prove that $\phi$ is identically zero. So for $\pi\in\hat{B}_n$, we have, using $r_n\circ s_n=Id_{\hat{B_n}}$ and formula (\ref{f}):
$$0=f(\phi)(s_n(\pi_n)) =\phi(s_n(\pi_n))-\phi(\pi_n),$$
i.e. $\phi(\pi_n)=\phi(s_n(\pi_n))$. Iterating, we have for every $k>0$:
$$\phi(s_{n+k}(...(s_n(\pi_n))))=\phi(\pi_n).$$
Since $\phi$ has finite support, for $k$ large enough we have $\phi(s_{n+k}(...(s_n(\pi_n))))=0$, hence $\phi(\pi_n)=0$.
\end{proof}

For $F$ a finite group, we denote by $Min(F)$ the set of equivalence classes of minimal idempotents in the group ring $\C F$. If we write $\C F=\bigoplus_{\pi\in\hat{F}} M_{\dim\pi}(\C)$ as above, $Min(F)$ can be realized by picking a rank one projection in each direct summand.

If $S$ is a pointed set, pointed by $o\in S$, we define $S^{(\Z)}$ as the set of maps $\Z\ra S$ taking the value $o$ for almost every integer, pointed by the constant map with value $o$.
If $F$ is a finite group, the dual $\hat{F}$ is pointed by the trivial representation $1_F$, and the set $Min(F)$ is pointed by the projection $p_F=\frac{1}{|F|}\sum_{g\in F} g$ associated with the trivial representation.

We come back to the setup of Proposition \ref{freeabelian1}. At the level of group rings, the analogue of the map $s_n$ of Proposition \ref{freeabelian1} is given by the {\it non-unital}, injective homomorphism
$$j_n:\C B_n\ra \C B_{n+1}=\C B_n\otimes\C F_{n+1}: x\mapsto x\otimes p_{F_{n+1}},$$
such that
$$(j_n)_*\circ\mu_0^{B_n}=\mu_0^{B_{n+1}}\circ s_n.$$
Then $j_n(Min(B_n))\subset Min(B_{n+1})$ and $K_0(C^*B)$ appears as the free abelian group on the co-limit of the system of sets $(Min(B_n),j_n)_{n>0}$.

When all $F_n$'s are equal to the same finite group $F$, our computations immediately yield:

\begin{Cor}\label{K0locfin} Let $F$ be a non-trivial finite group, and $B=\oplus_\Z F$. Then:
\begin{enumerate}
\item The free abelian group $\Z(\hat{F}^{(\Z)})$ is naturally isomorphic to $K_0^B(\uE B)$, with isomorphism
$$\pi\in\hat{F}^{(\Z)}\mapsto \otimes_{k\in supp\,\pi} \pi_k\in R_\C(\oplus_{k\in supp\,\pi}F_k)\subset K_0^B(\uE B)$$
(where $F_k$ denotes the $k$-th copy of $F$ in $B$).

\item  The free abelian group $\Z((Min(F))^{(\Z)})$ is naturally isomorphic to $K_0(C^*B)$, with isomorphism
$$e\in (Min(F))^{(\Z)}\mapsto [\otimes_{k\in supp\,e}e_k]\in Min(\oplus_{k\in supp\,e} F_k)\subset K_0(C^*B),$$
and in particular the constant map with value $p_F$ is mapped to $[1]$, the class of 1 in $C^*B$.
\item For $\pi\in\hat{F}^{(\Z)}$ we have $\mu_B(\pi)=\otimes_{k\in supp\,\pi}\mu_F(\pi_k)\in K_0(C^*B)$.
\end{enumerate}
\hfill$\square$
\end{Cor}

\section{Lamplighter groups over finite groups}

Let $F$ be a finite group, $B=\oplus_\Z F$ and $L=F\wr\Z=B\rtimes \Z$, where $\Z$ acts on $B$ by the shift.

\subsection{A model for $\uE L$}
\label{Fluch}
Let $B$ be a group, and $\alpha\in Aut(B)$ an automorphism. A result of M. Fluch (Proposition 5.11 in \cite{Flu11}) states that, if $B$ admits a $n$-dimensional model for $\uE B$, then the semi-direct product $L=B\rtimes\Z$ admits a $(n+1)$-dimensional $\uE L$. In particular, if $B$ is  a locally finite group, we observed in Section \ref{K-locfin} that $B$ admits a tree as $\uE B$; from Fluch's result, $L$ admits a 2-dimensional $\uE L$. We now recall Fluch's construction.

Let $B$ be the co-limit of an increasing sequence $(B_n)_{n>0}$ of finite groups. We recall that $\uE B$ is a tree $X$, with vertex and edge sets $V=E=\coprod_{n>0} B/B_n$, where the edge $xB_n$ connects the vertices $xB_n$ and $xB_{n+1}$ (for $x\in B$); there is an edge between cosets $xB_n$ and $yB_{n+1}$ if and only if $xB_n\subset yB_{n+1}$.
Let $\alpha$ be an automorphism of $B$; we assume that $\alpha^{-1}(B_n)\subset B_{n+1}$ for every $n$.

\begin{Ex} Let $\alpha$ be the shift on $B=\oplus_\Z F$, with $F$ a finite group. Set $B_n=:\oplus_{k=-n}^n F$: then $B$ is the co-limit of the $B_n$'s, and $\alpha^{-1}(B_n)\subset B_{n+1}$.
\end{Ex}

For $k\in\Z$, let $X_k$ be the tree $X$ with twisted $B$-action: $g\cdot_k xB_n=\alpha^{-k}(g)xB_n$, for $g,x\in B$. Let $f_k: X_k\rightarrow X_{k+1}$ be defined by $f_k(xB_n)=\alpha^{-1}(x)B_{n+1}$; it follows from our assumption that $f_k$ is well-defined; it is easily checked that $f_k$ is $B$-equivariant:
$$f_k(g\cdot_k xB_n)=g\cdot_{k+1}f_k(xB_n),$$
and $f_k$ maps edges of $X_k$ to edges of $X_{k+1}$.

Now we define the space $Y$ as a ``mapping telescope'':
$$Y=(\coprod_{k\in\Z} (X_k\times [k,k+1]))/\sim,$$
where $(xB_n,k+1)$ in the component $X_k\times [k,k+1]$ is identified to $(f_k(xB_n),k+1)$ in the component $X_{k+1}\times [k+1,k+2]$. Since $f_k$ is $B$-equivariant, the $B$-actions on the components glue together to yield a $B$-action on $Y$. We call the sub-complex $X_k$ the {\it $k$-th level} of $Y$.

We also have a ``vertical'' action of $\Z$ on $Y$ defined by $\sigma(xB_n,t)=(xB_n,t+1)$. It is easily checked that this action is well-defined, and free with fundamental domain $D=:X_0\times [0,1[$. Moreover, for $(xB_n,t)\in Y$ with $k\leq t\leq k+1$ and $g\in B$, we have:
$$\sigma\circ g\circ\sigma^{-1}(xB_n,t)=\sigma\circ g(xB_n,t-1)=\sigma (\alpha^{1-k}(g)xB_n,t-1)$$
$$=(\alpha^{1-k}(g)xB_n,t)=\alpha(g)(xB_n,t).$$
So $\sigma\circ g\circ\sigma^{-1}=\alpha(g)$, i.e. the $B$-action and $\Z$-action on $Y$ combine into an action of $L=B\rtimes\Z$. Since vertex stabilizers are clearly finite, the action of $L$ on $Y$ is proper.

\includegraphics[width=18cm,height=10cm]{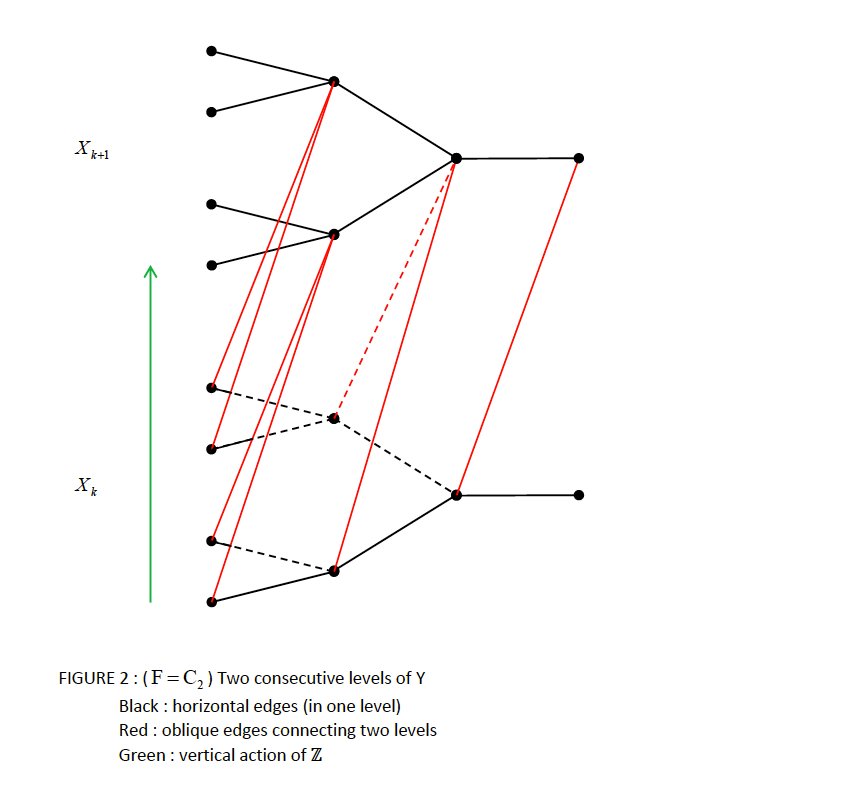}
\includegraphics[width=15cm,height=8cm]{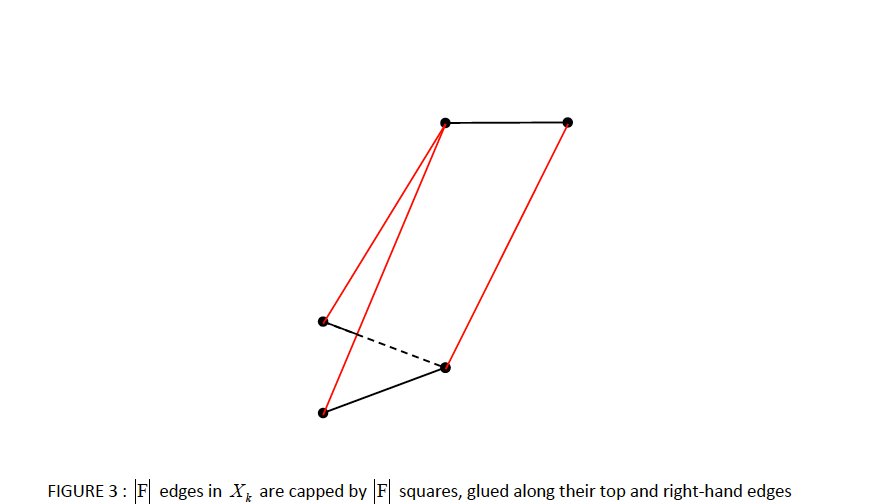}

\begin{Prop}\label{uEL} The space $Y$ is a 2-dimensional model for $\uE L$.
\end{Prop}

\begin{proof} Let $H$ be a subgroup of $L$. We must look at the fixed point set $Y^H$ and prove that it is empty if $H$ is infinite, and contractible if $H$ is finite. We consider two cases.
\begin{itemize}
\item $H$ is not contained in $B$ (implying that $H$ is infinite): then $Y^H=\emptyset$.
\item $H$ is contained in $B$. Then a point $(xB_n,t)\in\sigma^k(D)$ is $H$-fixed if and only if $H\subset \alpha^k(xB_n x^{-1})$. This already implies that $Y^H=\emptyset$ if $H$ is infinite.
Let $H$ be finite then. We follow Fluch's proof.  As $Y^H$ is a CW-complex, Whitehead Theorem implies that it is enough to show that it is weakly contractible, i.e. $\pi_n(Y^H)$ is trivial for every $n\geq 0$. Now $X^H$ is a subtree of $X$, hence is contractible, and so is $X_k^H=X^{\alpha^k(H)}$ for every $k\in\Z$.
 Hence, $Y^H$ is a mapping telescope of a collection of contractible spaces $X^H_k$. Now, every representative $S^n\rightarrow Y^H$ of an element on $\pi_n(Y^H)$ is contained, by compactness, in a finite subtelescope of $Y^H$, which retracts by deformation in $X_k^H$ for $k\gg 0$. Hence, $\pi_n(Y^H)$ is trivial and $Y^H$ weakly contractible, thus contractible.
\end{itemize}


\end{proof}

It follows from Proposition 2.3 in \cite{SV07} that $L=F\wr\Z$ admits a proper isometric action on a product of 2 trees (which is therefore a CAT(0) square-complex), but the example in Proposition \ref{uEL} is better adapted to the semi-direct product decomposition $L=B\rtimes\Z$.

\subsection{Computations on the left-hand side}

In this section we will compute Bredon homology with coefficients in the complex representation ring $R_{\mathbb{C}}$ of the infinite cyclic (split) extensions of countable locally finite of the form $G=B\rtimes \mathbb{Z}$, with $B$ countable locally finite. Our main tool will be Mart\'{i}nez spectral sequence, introduced in Section \ref{Martinez}, that we will use to compute Bredon homology of the extension $B\rightarrow G\rightarrow \mathbb{Z}$ with coefficients in the complex representation ring. In the notation of that section, $N=B$ is the countable locally finite group and $\tilde{G}$ are the integers. Again $\mathfrak{F}$ will be the family of finite subgroups of $G$, and $\tilde{\mathfrak{h}}$ the family of finite subgroups of $\mathbb{Z}$, that consists only in the trivial group. Observe that in this case the preimage of $\tilde{\mathfrak{h}}$ has only one element, identified with $B$. Now, in the notation above the $E_2$-page of the sequence is:

$$E_{p,q}^2=H_p(\mathbb{Z};H_q^{\mathfrak{F}_B}(B;R_{\mathbb{C}})).$$

Let us compute all the modules in the $E_2$-page.  First recall that the proper geometric dimension of $\mathbb{Z}$ is 1, as the real line is a minimal dimension model for $\underline{E}\mathbb{Z}$. Hence, $E_{p,q}^2$ should be trivial if $p\geq 2$. Moreover, by Theorem \ref{Mislin2}, Bredon homology of $B$ with coefficients in $R_{\mathbb{C}}$ is only non-trivial in dimension zero, and this implies that $E_{p,q}^2=0$ if $q\geq 1$. Hence, the page $E^2$ is concentrated in the vertical axis, and the spectral sequence collapses in this page, $E_{p,q}^2=E_{p,q}^{\infty}$. Hence we only need to compute $E_{0,0}^2$ and $E_{1,0}^2$, which in turn will be respectively isomorphic to $H_0^{\mathcal{F}}(G;R_{\mathbb{C}})$ and $H_1^{\mathcal{F}}(G;R_{\mathbb{C}})$.

\begin{Prop}
\label{invariants}
In the previous notation, $H_0^{\mathcal{F}}(G;R_{\mathbb{C}})$ is the space of co-invariants of the action of $\mathbb{Z}$ on $H_0^{\mathcal{F}_B}(B;R_{\mathbb{C}})$, and $H_1^{\mathcal{F}}(G;R_{\mathbb{C}})$  is the space of invariants. 

\end{Prop}

\begin{proof}

As stated above, $H_0^{\mathcal{F}}(G;R_{\mathbb{C}})$ and $H_1^{\mathcal{F}}(G;R_{\mathbb{C}})$ are respectively isomorphic to  $H_0(\mathbb{Z};H_0^{\mathfrak{F}_B}(B;R_{\mathbb{C}}))$ and $H_1(\mathbb{Z};H_0^{\mathfrak{F}_B}(B;R_{\mathbb{C}}))$, and the statements about co-invariants and invariants are a consequence of \cite{Bro82} (Example 1 in page 58), as we are dealing here with ordinary homology with coefficients.
\end{proof}

Now we can particularise our result for lamplighter groups.

\begin{Prop}\label{topolside}
If $L=F\wr\Z$ is the lamplighter group over the (non-trivial) finite group $F$, then $H_0^{\mathcal{F}}(L;R_{\C})$ is a free abelian group with a countable base indexed by the orbits of $\Z$ on $\hat{F}^{(\Z)}$, and $H_1^{\mathcal{F}}(L;R_{\mathbb{C}})=\mathbb{Z}$ is identified with the infinite cyclic subgroup in $\Z(\hat{F}^{(\Z)})$ generated by the base-point in $\hat{F}^{(\Z)}$. Moreover, these groups are respectively isomorphic to $ K_0^L(\underline{E}L)$ and $ K_1^L(\underline{E}L)$.
\end{Prop}

\begin{proof} By the computations involving Mart\'{i}nez spectral sequence, we have $H_2^{\mathcal{F}}(L;R_{\mathbb{C}})=0$; by Theorem \ref{Mislin2}, we then have $H_i^{\mathcal{F}}(L;R_{\mathbb{C}})\simeq K_i^L(\uE L)$ for $i=0,1$. By Proposition \ref{invariants}, these groups correspond respectively to the co-invariants and invariants of $\Z$ acting on $H_0^{\mathcal{F}}(B;R_{\mathbb{C}})$; and Corollary  \ref{K0locfin} implies that the latter is naturally isomorphic to $\Z(\hat{F}^{(\Z)})$), with $\Z$ acting by shift. The result then follows from Lemma \ref{inv/coinv}, observing that the shift has a unique finite orbit on $\hat{F}^{(\Z)}$, namely the fixed point provided by the base-point $o$.
\end{proof}

\begin{Rem}

It is interesting to point out here another result by P. Baum and A. Connes (\cite{BC88}, Proposition 15.2) where, for every discrete countable group $G$, it is constructed a Chern character $Ch_*: K_i^G(\uE G)\rightarrow \bigoplus_{j\geq 0} H_{i+2j}(G,FG)$ where $FG$ is the complex vector space generated by finite order elements in $G$, with $G$ acting by conjugation; $Ch_*$ becomes an isomorphism after tensoring by $\C$. An alternative construction of W. L\"uck (Theorem 0.7 in \cite{Luc02}) provides, after inverting orders of finite subgroups of $G$, an identification between $K_i^G(\uE G)$ and the direct sum, over conjugacy classes of finite cyclic subgroups $C$ of $G$, of abelian groups which are basically (non-equivariant) K-homology groups of the fixed point point set $(\uE G)^C$. All this makes $K_i^G(\uE G)$ computable rationally, and this can be used as a test case before undertaking computations with the spectral sequence, which are usually more complicated. See also the discussion in (\cite{MV03}, I.3.25).

\end{Rem}

\subsection{Computations on the right-hand side}
In this section we present the computations for the \emph{K}-groups  $K_i(C^*L)$, $i=0, 1$,  for $L=B\rtimes \Z $, and $B=F^{(\Z)}$, $F$ a (non-trivial) finite group. Recall that for every amenable group $G$, we have $C^*G=C^*_rG$.

Let $A$ be a unital $C^*$-algebra on which $\Z$ acts via $\alpha \in Aut(A)$. We may construct $A\rtimes_{\alpha}\Z$, with $u \in A\rtimes_{\alpha}\Z$ be the unitary which implements the action in this construction. Abstractly,
$A\rtimes_{\alpha}\Z$ is the universal $C^*$-algebra generated by $A$ and $u$ corresponding to the relation $\alpha(a)=uau^{-1}$ for $a\in A$. In our case  $C^*(L) = C^*(B)\rtimes_{\alpha} \Z $, where $\alpha$ denotes the shift on $C^*B$.

A very convenient tool for computing the \emph{K}-theory of crossed products by $\Z$ is the Pimsner-Voiculescu 6-term exact sequence (P-V), see Theorem 2.4 in \cite{PV80}:
\begin{equation*}
\begin{array}{ccccc}
K_0(A) 	& \stackrel{Id-\alpha_*}{\longrightarrow} & K_0(A) & \stackrel{\iota_*}{\longrightarrow} & K_0(A\rtimes_{\alpha} {\mathbb Z}) \\
& & & &  \\
 \partial_1\uparrow &   &   &   & \downarrow \partial_0\\
 & & & &  \\
K_1(A\rtimes_{\alpha} {\mathbb Z}) & \stackrel{\iota_*}{\longleftarrow} & K_1(A) & \stackrel{Id-\alpha_*}{\longleftarrow} & K_1(A)
\end{array}
\end{equation*}
Let $[u]\in K_1(A\rtimes_\alpha\Z)$ be the \emph{K}-theory class of the basic unitary $u\in A\rtimes\Z$ generating the $\Z$-action on $A$.

\begin{Thm} {\label {K-theory for $C^*L$}}
	Let $L = B\rtimes_{\alpha} \Z$, with $B=\oplus_\Z F$, $F$ a non-trivial finite group. Then $K_0(C^*L)$ identifies with the free abelian group on the orbit space of the shift on $Min(F)^{(\Z)}$, and $K_1(C^*L)$ is the infinite cyclic group generated by $[u]$.
\end{Thm}

\begin{proof}
 We plug $C^*B$ and $C^*L$ in the P-V exact sequence; by Corollary \ref{K0locfin}, $K_0(C^*B)$ is identified, $\alpha$-equivariantly, with $\Z((Min(F))^{(\Z)})$. Since $C^*B$ is an AF-algebra, its $K_1$-group is zero.
\begin{equation*}
\begin{array}{ccccc}
	\Z((Min(F))^{(\Z)})	& \stackrel{Id- \alpha _*}{\longrightarrow} & \Z((Min(F))^{(\Z)}) & \stackrel{\iota_*}{\longrightarrow} & K_0( C^*(L)) \\ \partial_1\uparrow &   &   &   & \downarrow \\
	K_1( C^*L) & {\longleftarrow} & 0 & {\longleftarrow} & 0
\end{array}
\end{equation*}
By injectivity of $\partial_1$ and exactness of the sequence $K_1(C^*L)\simeq Ker (Id- \alpha_*) = \Z.p_F$, where $p_F$ denotes the constant map taking the value $p_F$ on $\Z$; it corresponds to the class $[1]$ of 1 in $K_0(C^*B)$ (see Corollary \ref{K0locfin}). Moreover we know from Lemma 2 in \cite{PV16} that $\partial_1[u] = -[1]$. Hence $K_1(C^*L)= \Z.[u]$. In order to compute $K_0(C^*L)$, we focus on the right side of the diagram. Surjectivity of $\iota_*$ implies that $K_0(C^*L)\simeq \Z((Min(F))^{(\Z)})/ {Ker{\iota_*}}$. By exactness of the diagram we get $Ker{\iota_*}= Im (Id-\alpha_*)$. Hence $K_0(C^*L)$ is the set of co-invariants of the shift in $\Z((Min(F))^{(\Z)})$, which is described by Lemma \ref{inv/coinv} and Lemma \ref{repres}.
\end{proof}

\subsection{The Baum-Connes conjecture and the trace conjecture}

Now we will describe explicitly Baum-Connes isomorphism.

\begin{Thm}\label{BCforL} Let $L=F\wr\Z$ be the lamplighter group over some (non-trivial) finite group $F$.
\begin{enumerate}
\item[a)] The map $\mu_0^L:K_0^L(\uE L)\ra K_0(C^*L)$ is an isomorphism between two free abelian groups of countably infinite rank.
\item[b)] The map $\mu_1^L:K_1^L(\uE L)\ra K_1(C^*L)$ is an isomorphism between two infinite cyclic groups. More precisely, denoting by $i:\Z\rightarrow L$ the inclusion, we have a commuting diagram of isomorphisms:
$$\begin{array}{ccc}\Z=K_1^\Z(\uE Z) & \stackrel{\mu_1^\Z}{\rightarrow} & K_1(C^*\Z)=\Z \\i_*\downarrow &  & \downarrow i_* \\K_1^L(\uE L) & \stackrel{\mu_1^L}{\rightarrow} & K_1(C^*L)\end{array}$$
\end{enumerate}
\end{Thm}

\begin{proof}
\begin{enumerate}
\item Case of $\mu_0^L$. We already know that $K_0^L(\uE L),K_0(C^*L)$ are free abelian on countably many generators. Let $\iota$ denote inclusion of $B$ into $L$. Consider the following diagram:
$$\begin{array}{ccccccc}
K_0^B(\uE B) & \stackrel{Id-\alpha_*}{\ra} & K_0^B(\uE B) & \stackrel{\iota_*}{\ra} & K_0^L(\uE L) & \ra & 0 \\
 & & & & & & \\
\mu_0^B\downarrow & & \mu_0^B\downarrow & & \mu_0^L\downarrow & & \\
 & & & & & & \\
K_0(C^*B) & \stackrel{Id-\alpha_*}{\ra} & K_0(C^*B) & \stackrel{\iota_*}{\ra} & K_0(C^*L) & \ra & 0
\end{array}$$
This diagram commutes, by functoriality of the assembly map (see \cite{MV03}, Corollary II.1.3). By Proposition \ref{topolside} and his proof (resp. by Theorem \ref{K-theory for $C^*L$} and its proof), the top line (resp. bottom line) is exact. Since $\mu_0^B$ is an isomorphism (as we already mentioned), $\mu_0^L$ is an isomorphism too, by the suitable version of the Five Lemma.

\item Case of $\mu_1^L$.  The diagram commutes, again by functoriality of the assembly map. We must show that the left hand-side vertical and bottom horizontal maps are isomorphisms. Let $p:L\ra\Z$ be the quotient map, and let $t$ be the generator of $K_1^\Z(\uE \Z)\simeq\Z$ as described in \cite{MV03}, section II.4.2. We know by Proposition \ref{topolside} that $K_1^L(\uE L)$ is infinite cyclic. Denote by $u$ any of its two generators. Since $p_*\circ i_*$ is the identity of $K_1^\Z(\uE \Z)$, we know that $i_*(t)$ is non-zero in $K_1^L(\uE L)$. Let $k$ be the non-zero integer such that $i_*(t)=k.u$. By commutativity of the diagram we have:
$$k.\mu_1^L(u)=\mu_1^L( i_*(t))=i_*(\mu_1^\Z(t)).$$
But $\mu_1^\Z(t)$ is the canonical generator of $K_1(C^*(\Z))=K_1(C(S^1))=\Z$ (see \cite{MV03}, Proposition II.4.1), and by Theorem \ref{K-theory for $C^*L$} the inclusion $i:C^*(\Z)\ra C^*(L)$ induces an isomorphism at the level of $K_1$. So $i_*(\mu_1^\Z(t))$ is a generator of $K_1(C^*(L))=\Z$, which implies $k=\pm 1$ and both $i_*$ and $\mu_1^L$ are isomorphisms.
\end{enumerate}

\end{proof}
 Now we present the proof of the trace conjecture for $C^*(L)$.

 \begin{Prop}\label{trace}
 	For $L= F \wr \Z$, 	Let $\tau : C^*L \rightarrow \mathbb C$ be the canonical trace. The induced homomorphism
 	$\tau _* : K_0( C^*L) \rightarrow \mathbb R$
 	at the \emph{K}-theory level has  $\Z[\frac {1}{|F|}]$ as its image.
 \end{Prop}
\begin{proof}
	
	Let $L = B\rtimes_{\alpha} \Z$, and let $\tau\colon C^*L \rightarrow \mathbb C$ be the canonical trace on $C^*L$. For $B$ the subgroup of $L$, the restriction of $\tau$ on $C^*B$, still denoted by $\tau$, gives the canonical trace on $C^*B$.
	Surjectivity of $\iota _* \colon K_0(C^*B) \rightarrow K_0(C^* L)$, as in Theorem \ref{K-theory for $C^*L$}, implies that $\tau_*(K_0(C^*L)) = \tau_*(K_0(C^* B))$. Hence we need to show $\tau_*(K_0(C^*B)) = \Z[\frac {1}{|F|}] $.
	\\
	View $B$ as $B = \varinjlim {B_n}$ with $B_n$'s finite groups. We obtain the  inductive limit  $C^*B = \varinjlim\C B_n$. Since \emph{K}-theory commutes with  inductive limits, we have $K_0(C^*B) = \varinjlim K_0(\C B_n)$. Similar to before the canonical trace $\tau$ on $C^*B$ restricts to the canonical trace $\tau$ on $\C B_n$'s for  $n\in \N$. Therefore $\tau_* (C^* B)= \varinjlim (\tau_*(K_0(\C B_n)), \iota_n)$, with $\iota_n$'s  inclusions the  $\iota_n\colon \tau_*(K_0(\C B_{n})) \rightarrow \tau_*(K_0(\C B_{n+1}))$ for $n\in \N$.  This way we reduce to finite groups; the trace conjecture for finite groups is of course well-known, we give below a proof for completeness. It immediately implies that $\tau_* (K_0(C^*(L)) = \varinjlim {( \frac{1}{|F|^n} {\Z}, \iota_{n})} = \Z [\frac{1}{|F|}]$.
	
	\medskip
	{\bf Claim:} For a finite group $H$, consider $\tau_*\colon K_0(\C H) \rightarrow \R$. Then $\operatorname{Im} \tau_* =  \frac{1}{|H|} {\Z}$.

	\medskip
	{\bf Proof of the claim:} Consider the minimal projection $p_H = \frac{1}{|H|}\sum_{h\in H}{h}$ associated to the trivial representation of $H$ as in Section \ref{K-locfin}. Since $\tau(p_H) = \frac{1}{|H|}$
	we get that $\frac{1}{|H|}  \Z \subset \tau_*(K_0(\C H))$. To show the converse, for $\pi \in \hat H$ consider the character associated to $\pi$, i.e. $\chi _{\pi} (h) = \operatorname{Tr}(\pi(h))$ for $h\in H$. Consider first $p_{\pi} = \frac{\dim \pi}{|H|} \chi _{\pi}$ the minimal central projection in $\C H$ such that $\C H = \oplus_{\pi \in {\hat H}} p_{\pi} \C H \cong \oplus_{\pi \in {\hat H}} M_{\dim {\pi}}(\C)$. They satisfy $\tau (p_{\pi}) = \frac{(\dim \pi )^2}{|H|}$. Now every $p_\pi$ is the sum of $\dim\pi$ minimal projections in $M_{\dim\pi}(\C)$, all with the same trace; hence, the trace of a minimal projection in $\C H$ is $\frac{\dim\pi}{|H|}\in\frac{1}{|H|} \Z.$ This finishes the proof.
	\hfill$\square$

\end{proof}

\section{Applications to full shifts}

Suppose that $F$ is a (non-trivial) finite {\it abelian} group. Then $B=\oplus_\Z F$ is a countable abelian group, so $C^*B$ identifies via Fourier transform with $C(\hat{B})$, where $\hat{B}$ is the Pontryagin dual of $B$, here $\hat{B}=\hat{F}^\Z$, a Cantor set.

Let $Y$ be a Cantor set; it is well-known that $K_0(C(Y))=C(Y,\Z)$, the integer-valued continuous functions on $Y$. Let $\phi$ be a homeomorphism of $Y$, for instance the shift on $\hat{B}$ (called the {\it full shift} in topological dynamics). It was first observed by Putnam in Theorem 1.1 of \cite{Put89} that  $K_1(C(Y)\rtimes_{\phi} \Z) = <[u]>$ and $K_0(C(Y)\rtimes_{\phi} \Z) \cong C(Y, \Z)/ {Im(Id - \phi ^*)}$ , as a consequence of the P-V sequence\footnote{The isomorphism $K_0(C(Y)\rtimes_{\phi} \Z) \cong C(Y, \Z)/ {Im(Id - \phi ^*)}$ is indeed an isomorphism of {\it preordered} groups, by Theorem 5.2 in \cite{BH96}.}. For the full shift it follows from our Theorem \ref{K-theory for $C^*L$}.

It is probably folklore that $C(Y,\Z)$ is a free abelian group on countably many generators. To apply our results to full shifts, we need an explicit identification, so we spell out
Corollary \ref{K0locfin} for $K_0(C^*B)$ in the case where $F$ is abelian.

We may identify $Min(F)$ with $\hat{F}$, the character group of $F$. For $\varepsilon\in\hat{F}^{(\Z)}$, let $\chi_\varepsilon$ be the characteristic function of the clopen set $\{x\in \hat{B}: x_k=\varepsilon_k,\,\mbox{for every}\,k\in\,supp\,\varepsilon\}$. Note that, if $1$ denotes the constant map with value $1$ (the trivial character of $B$), then $\chi_1$ is the constant function 1. The following is then a re-writing of Corollary \ref{K0locfin}:

\begin{Prop}\label{C(Y,Z)free} The map
$$\beta: \Z(\hat{F}^{(\Z)})\rightarrow C(\hat{B},\Z):\sum_{\varepsilon\in\hat{F}^{(\Z)}}a_\varepsilon \varepsilon\mapsto \sum_{\varepsilon\in\hat{F}^{(\Z)}} a_\varepsilon \chi_\varepsilon$$
is a (shift-equivariant) isomorphism.
\hfill$\square$
\end{Prop}

Observe that the group structure of $\hat{F}$ is irrelevant in Proposition \ref{C(Y,Z)free}. We only use the fact that it is a finite pointed space.

Denoting again by $\alpha$ the shift on $C(\hat{B},\Z)$, we denote by $B^1(\hat{B},\alpha)=\{f-f\circ\alpha: f\in C(\hat{B},\Z)\}$ the space of 1-coboundaries, and by $H^1(\hat{B},\alpha)=C(\hat{B},\alpha)/B^1(\hat{B},\alpha)$ the corresponding 1-cohomology group. Set $L=F\wr\Z$. Using the identification between $K_0(C^*L)$ and $H^1(\hat{B},\alpha)$ given by Theorem \ref{K-theory for $C^*L$} and its proof, we see that $H^1(\hat{F}^Z,\alpha)$ is free abelian on countably many generators. This result was known (see \cite{PT82}, Theorem 19 in Chapter 4), but our approach allows us to give an explicit basis, as we now explain\footnote{The original argument consists in expressing $H^1(\hat{B},\alpha)$ as a co-limit of free abelian groups of finite rank, such that the consecutive quotients are torsion-free.}.

For  $A\subset \Z$, we say that a function $f\in C(\hat{B},\Z)$ {\it only depends only the coordinates in $A$} if it factors through the restriction map $\hat{B}\rightarrow \hat{F}^{A}$. Define $$\tilde{H}(\hat{B}, \Z) =
  \{f\in C(\hat{B}, \Z) \,|\,\, f \,\,\text{only depends on the coordinates in} \,\,\N \cup \{0\}$$
   $$\mbox{and } \,\,f(x)=0 \,\, \,\,\mbox{if}\,\, x_0=1\}.$$
Recall that, for $g,h\in\hat{F}\backslash\{1\}$ and $\varepsilon\in\hat{F}^n$, we defined elements $g$ and $g\varepsilon h$ in $\hat{F}^{(\Z)}$, before Lemma \ref{repres}:
$$g(k)=\left\{\begin{array}{ccc}g & if & k=0 \\1 & if & k\neq 0\end{array}\right.$$
$$(g\varepsilon h)(k)=\left\{\begin{array}{ccc}g & if & k=0 \\\varepsilon_k & if & k=1,...,n \\h & if & k=n+1 \\1 &  & otherwise\end{array}\right.$$
The set of all those elements of $\hat{F}^{(\Z)}$ is denoted by $S$.

\begin{Lem}{\label{generatorHtilde}}
	 $ \tilde{H}(\hat{B}, \Z)$ is freely generated by the set
	 {$\mathcal{B} = \{\chi_g, \chi_{g\varepsilon h} \,|\,\, \varepsilon  \in \hat{F}^n, n \in \mathbb N \,\, \text{and} \,\,g, h\neq 1\}$.}
\end{Lem}
 \begin{proof}
 	Clearly $\mathcal B \subset\tilde{H}(X, \Z) $. By Proposition \ref{C(Y,Z)free}, the set $\mathcal{B}$ is the image under the map $\beta$ of the free set $S\backslash\{1\}$ in $\Z(\hat{F}^{(\Z)})$, so it is a free subset of $ \tilde{H}(\hat{B}, \Z)$, and it is enough to prove that $\mathcal{B}$ generates $ \tilde{H}(\hat{B}, \Z)$. We know that $\tilde{H}(\hat{B}, \Z)$ is generated by the set  $\{\chi_{g\varepsilon}\,\,|\,\, g\in \hat{F}\setminus \{1\}, \varepsilon \in F^n, n\in \N \}$.  We would like to show that we may reduce the number of generating elements by putting aside those clopen sets whose associated word $g\varepsilon h$ ends in $1$, by producing them from others. We show by induction on the length of word $\varepsilon h$ that all such $\chi_{g\varepsilon h}$'s are in the span of $\mathcal B$. Let $n=0$, then $\chi_g \in \mathcal B$ by definition. Assume the result is true for $n = m$. For the case $n= m+1$, if $h \neq 1$, then $\chi_{g\varepsilon h} \in \mathcal B$; if $h=1$ then write $\chi_{g\varepsilon 1}$ as the following combination
 	$$
 	\chi_{g\varepsilon e} = \chi_{g \varepsilon} - \sum_{k\neq 1}{\chi_{ g\varepsilon k}}
 	$$
 	Here  the word associated to $\chi_{g\varepsilon}$ has length $m$ so by induction hypothesis it is in $\mathcal B$, and the  words in the sum do not end in 1 hence they are in $\mathcal B$ as well.
 	 \end{proof}
	
	 Set then $H(\hat{B},\Z)=\tilde{H}(\hat{B},\Z)+\Z.1$. Next result identifies the cohomology group $H^1(\hat{B},\alpha)$.
	
	 \begin{Prop}\label{directsum} $H(\hat{B},\Z)$ is the free abelian group on $\mathcal{B}\cup\{1\}$, and $C(\hat{B},\Z)$ is the direct sum of $B^1(\hat{B},\Z)$ and $H(\hat{B},\Z)$.
	 \end{Prop}
	
	 \begin{proof} The first statement follows immediately from Lemma \ref{generatorHtilde}. For the second: we observe that, by Proposition \ref{C(Y,Z)free}, $\mathcal{B}\cup\{1\}=\beta(S)$. Since
	 $$\Z(\hat{F}^{(\Z)})= <m-\alpha(m):m\in\Z(\hat{F}^{(\Z)})>\oplus\; \Z(S)$$
	 by Lemma \ref{repres}, we have $C(\hat{B},\Z)=B^1(\hat{B},\alpha)\oplus H(\hat{B},\Z)$ using Lemma \ref{generatorHtilde}.
	 \end{proof}
	
	 We end with a characterization \emph{\`a la} Liv\u{s}ic of the coboundaries in $B^1(\hat{B},\alpha)$. It is worth noting that this result can be deduced from the more general Theorem 4 in \cite{Liv72}, and is also stated in Proposition 3.13.(b) of \cite{BH96}. We provide an almost entirely algebraic proof.
	
	 \begin{Thm}\label{Livsic} For $f\in C(\hat{B},\Z)$, the following are equivalent:
	 \begin{enumerate}
	 \item $f\in B^1(\hat{B},\alpha)$;
	 \item $\int_{\hat{B}} f\,d\mu=0$ for every $\alpha$-invariant probability measure $\mu$ on $\hat{B}$;
	 \item For every $n\geq 1$ and every $n$-periodic $x\in\hat{B}$,
	 $$\sum_{k=0}^n f(\alpha^k(x))=0$$
	 (i.e. $f$ sums to 0 on the orbit of every periodic point).
	 \end{enumerate}
	 \end{Thm}
	
\begin{proof} The implications $(1)\Rightarrow(2)\Rightarrow(3)$ are clear. To prove $(3)\Rightarrow(1)$, we take $f\in C(\hat{B},\Z)$ summing to zero on all periodic orbits, and must prove that $f$ is a coboundary. In view of Proposition \ref{directsum}, we may assume $f\in H(\hat{B},\Z)$ and should check that $f$ vanishes identically. For $k\in\Z$, set $f_k=:f\circ\alpha^k$. Write $f=\sum_{\varepsilon\in S} a_\varepsilon\chi_\varepsilon$ (with $a_\varepsilon\in\Z$, vanishing for all but finitely many). Then $f_k=\sum_{\varepsilon\in S} a_\varepsilon\chi_{\alpha^{-k}(\varepsilon)}=\sum_{\eta\in\alpha^{-k}(S)} a_{\alpha^k(\eta)}\chi_\eta$. By Lemma \ref{repres}, for $k\neq\ell$, the sets $\alpha^{-k}(S)$ and $\alpha^{-\ell}(S)$ are disjoint, and then by Proposition \ref{C(Y,Z)free} $f_k$ and $f_\ell$ have disjoint supports. Now let $x\in\hat{B}$ be $n$-periodic. By assumption, $\sum_{k=0}^{n-1}f_k(x)=0$, and by disjointness of the supports of the $f_k$'s we deduce $f_k(x)=0$ for $k=0,1,...,n-1$. In particular, $f=f_0$ vanishes on any periodic point in $\hat{B}$. Since periodic points are dense in $\hat{B}$, the function $f$ is identically zero.
\end{proof}

If $\beta$ is a homeomorphism of a Cantor set $X$, the {\it space of infinitesimals} of $H^1(X,\beta)$ is the set of elements of $H^1(X,\beta)$ annihilated by all $\beta$-invariants functionals on $C(X,\R)$. Theorem \ref{Livsic} can then be rephrased by saying that the group $H^1(\hat{B},\alpha)$ has no infinitesimals, and this is of course related to the fact that uniform probability measures on periodic orbits of $\hat{B}$ are exactly the ergodic probability measures of the full shift (see Theorem 2.3 in \cite{Po89}).

\section{Distinguishing $C^*L$ up to isomorphism}

In this section we write $L_F=F\wr\Z$ (rather than $L$) to emphasize the dependency on $F$.

It follows from our \emph{K}-theory computations (Theorem \ref{K-theory for $C^*L$}) that the $C^*$-algebras $C^*(L_F)$ cannot be distinguished by $K$-theory. This raises the natural question of distinguishing them up to $*$-isomorphism. Our results, although partial, give however a satisfactory answer when $F$ is abelian.

As we observed already, $C^*B=\otimes_{k\in\N}\C F$ is an AF-algebra. Its Bratteli diagram is a rooted tree, obtained by repeating at each vertex the diagram of the inclusion $\C\subset \C F$:

\bigskip
\begin{picture}(120,60)
\put(55,60){$\bullet$}
\put(70,60){$\C$}
\put(0,0){\mbox{$\C$}}
\put(40,0){\mbox{$M_{d_2}(\C)$}}
\put(90,0){\mbox{$\cdots$}}
\put(140,0){\mbox{$M_{d_r}(\C)$}}
\put(3,12){$\bullet$}
\put(55,12){$\bullet$}
\put(150,12){$\bullet$}
\put(57,13){\line(0,1){48}}
\put(5,13){\line(1,1){53}}
\put(59,62){\line(2,-1){95}}
\put(115,35){$d_r$}
\put(60,35){$d_2$}
\put(-10,35){$d_1=1$}
\end{picture}

\medskip
Here the bottom vertices are indexed by $\hat{F}=\{\pi_1=1_F,\pi_2,...,\pi_r\}$ and the number of edges between the top vertex and the bottom vertex $\pi_i\in\hat{F}$ is the dimension $d_i=\dim\pi_i$.

So the Bratteli diagram of $C^*B$ only depends on the group ring $\C F$, and then, if $F_1,F_2$ are finite groups with isomorphic group rings, then $C^*L_{F_1}\simeq C^*L_{F_2}$. The question now is if the converse will be also true, i.e. if the $C^*$-algebra of $L_F$ determines the group ring of $F$. We will prove that the answer is affirmative when one of the groups $F_1,F_2$ is abelian (recall that, for abelian groups: $\C F_1\simeq\C F_2$ if and only if $|F_1|=|F_2|$).

\begin{Lem}\label{resfin} The following are equivalent:
\begin{enumerate}
\item $F$ is abelian;
\item $C^*L_F$ is residually finite-dimensional.
\end{enumerate}
\end{Lem}

\begin{proof} $(1)\Rightarrow(2)$ Assume that $F$ is abelian. Then $L_F$ is residually finite (the quotient maps onto the finite groups $F\wr C_n$ ($n\geq 2$) do separate points), and the $C^*$-algebra of an amenable, residually finite group, is residually finite-dimensional (see Theorem 4.3 in \cite{BL00}).

$(2)\Rightarrow(1)$ 
If $C^*L_F$ is residually finite-dimensional, then $L_F$ is maximally almost periodic (i.e. finite-dimensional unitary representations separate points); but $L_F$ is also finitely generated. By Selberg's lemma (finitely generated linear groups are residually finite), we deduce that $L_F$ is residually finite. Now a result of Gruenberg (Theorem 3.2 in \cite{Gr57}) implies that $F$ is abelian.
\end{proof}


Let $A$ be a $C^*$-algebra. We denote by $A^{ab}$ the {\it abelianization} of $A$, i.e. the greatest abelian $C^*$-quotient of $A$; this is the quotient of $A$ by the closed 2-sided ideal generated by commutators in $A$. If $A=C^*G$ for some discrete group $G$, then $(C^*G)^{ab}=C^*(G^{ab})$ (where $G^{ab}$ is the abelianization of $G$): this follows easily from the universal property of $C^*G$.

\begin{Lem}\label{abelianize} Set $n=|F^{ab}|$. Then $(C^*(L_F))^{ab}\simeq \C^n\otimes C(S^1)$.
\end{Lem}

\begin{proof} By the previous observation we have
$$(C^*(L_F))^{ab}=C^*(L_F^{ab})=C^*(F^{ab}\times\Z)\simeq \C^n\otimes C(S^1).$$

\end{proof}

\begin{Thm}\label{uptoisom} Let $F_1,F_2$ be finite groups, with $F_1$ abelian. The following are equivalent:
\begin{enumerate}
\item $C^*(L_{F_1})\simeq C^*(L_{F_2})$;
\item $F_2$ is abelian and $|F_1|=|F_2|$.
\end{enumerate}
\end{Thm}

\begin{proof} $(2)\Rightarrow(1)$. If $F_2$ is abelian of the same order as $F_1$, then $\C F_1\simeq \C F_2$ and hence $C^*(L_{F_1})\simeq C^*(L_{F_2})$ by a previous observation.

$(1)\Rightarrow(2)$. Assume $C^*(L_{F_1})\simeq C^*(L_{F_2})$. By Lemma \ref{resfin}, $F_2$ is abelian, as $C^*(L_{F_1})$ is residually finite-dimensional. By Lemma \ref{abelianize}, the order of $F^{ab}$ is recovered from $C^*(L_F)$, so in our case $|F_1|=|(F_2)^{ab}|=|F_2|$.
\end{proof}

The {\it stable rank} of a $C^*$-algebra $A$ is a number $sr(A)\in\{\infty,1,2,3,...\}$ that was introduced by Rieffel \cite{Rie83}. An important feature is that, for $A$ unital, $sr(A)=1$ if and only if the invertible group $GL(A)$ is dense in $A$. We extend a result of Poon (\cite{Po89-b}, Corollary 2.6) from $F$ abelian to $F$ arbitrary finite group.

\begin{Prop}\label{stablerank} In the previous notation, $sr(C^*(L_F))=2$.
\end{Prop}

\begin{proof} Write $C^*(L_F)=C^*B\rtimes_\alpha\Z$. By Theorem 7.1 in \cite{Rie83}: $sr(C^*(L_F))\leq sr(C^*B)+1.$ Now $sr(C^*B)=1$ as $C^*B$ is an AF-algebra (\cite{Rie83}, Proposition 3.5). So
$$sr(C^*(L_F))\leq 2$$
and we are left to prove that $sr(C^*(L_F))>1$.

We will make use of a result of Poon (Lemma 2.2 in \cite{Po89-b}): if $I$ is a closed 2-sided ideal in a $C^*$-algebra $A$, we have $sr(A)=1$ if and only if $sr(I)=sr(A/I)=1$ and moreover the map $K_1(A)\ra K_1(A/I)$ is onto. So it is enough to find a quotient of $C^*(L_F)$ such that the map $K_1(C^*(L_F))\ra K_1(C^*(L_F)/I)$ is {\it not} onto.

Let $\pi$ be a non-trivial irreducible representation of $F$, say of dimension $d\geq 1$. Choose some unit vector $\xi$ in the Hilbert space $\mathcal{H}_{\pi}$. We then have an isometric embedding of finite-dimensional Hilbert spaces:
$$\bigotimes_{k=-n}^n \mathcal{H}_{\pi}\ra\bigotimes_{k=-n-1}^{n+1} \mathcal{H}_{\pi}:x\mapsto \xi\otimes x\otimes\xi.$$
The inductive limit of this system of Hilbert spaces is a Hilbert space $\bigotimes_{k\in\Z}(\mathcal{H}_{\pi},\xi)$.
We also have unitary representation $\bigotimes_{k=-n}^n \pi$ of $B_n=\oplus_{k=-n}^n F$ on $\bigotimes_{k=-n}^n \mathcal{H}_{\pi}$, that we embed as a sub-representation of $\bigotimes_{k=-n-1}^{n+1} \pi$ by $S \mapsto Id_{\mathcal{H}_\pi}\otimes S \otimes Id_{\mathcal{H}_\pi}$. Passing to the inductive limit, we get a unitary representation $\rho=:\bigotimes_{k\in\Z} (\pi,\xi)$ of $B$ on $\bigotimes_{k\in\Z}(\mathcal{H}_{\pi},\xi)$, that is irreducible by Section III.3.1.2 in \cite{Bla06}.

Set $\sigma=\rho\oplus 1_B$. At the $C^*$-level, $\sigma$ induces a short exact sequence
$$0\ra \ker\sigma\ra C^*B\ra \rho(C^*B)\oplus \C\ra 0$$
which is $\alpha$-equivariant. Taking crossed products by $\Z$, we get a short exact sequence
$$0\ra (\ker\sigma)\rtimes_\alpha\Z\ra C^*(L_F)\ra (\rho(C^*B)\rtimes_\alpha\Z)\oplus C^*(\Z)\ra 0.$$
The proof will completed by showing that the induced map $K_1(C^*(L_F))\ra K_1(\rho(C^*B)\rtimes_\alpha\Z)\oplus K_1( C^*(\Z))$ is not onto. Since $K_1(C^*(L_F))=\Z=K_1(C^*\Z)$, it is enough to see that $K_1(\rho(C^*B)\rtimes_\alpha\Z)\neq 0$. If $d=1$, then $\rho(C^*B)=\C$ and $\rho(C^*B)\rtimes_\alpha\Z=C^*(\Z)$, so the result is easy. If $d\geq 2$, then $\rho(C^*B)=\otimes_{k\in\Z} M_d(\C)$, i.e. $\rho(C^*B)$ is the UHF algebra $M_{d^\infty}(\C)$. We then consider the connecting map $\partial: K_1(\rho(C^*B)\rtimes_\alpha\Z)\ra K_0(\rho(C^*B))$ in the Pimsner-Voiculescu exact sequence. Let $[u]\in K_1(\rho(C^*B)\rtimes_\alpha\Z)$ be the unitary implementing the shift. By Lemma 2 in \cite{PV16}, we have $\partial[u]=-[1]$ and $[1]\neq 0$ as $M_{d^\infty}(\C)$ carries a unital trace. (Slightly more work gives $K_1(\rho(C^*B)\rtimes_\alpha\Z)=\Z[\frac{1}{d}]$.) This completes the proof.
\end{proof}

The {\it real rank} of a $C^*$-algebra $A$ was introduced by Brown and Pedersen \cite{BP91}; it is a number $RR(A)\in\{\infty, 0, 1, 2,\ldots\}$. A unital $C^*$-algebra satisfies $RR(A)=0$ if and only if invertible self-adjoint elements are dense among self-adjoint elements in $A$. The condition $RR(A)=0$ is clearly inherited by quotients, so a $C^*$-algebra with $RR(A)=0$ cannot map onto $C(S^1)$. This already shows $RR(C^*(L_F))>0$. On the other hand, for any $C^*$-algebra $A$, we have $RR(A)\leq 2sr(A)-1$ (Proposition 1.2 in \cite{BP91}). From Proposition \ref{stablerank} we deduce: $RR(C^*(L_F))\in \{1,2,3\}$. We therefore finish this paper by pointing to an interesting problem.
\vspace{.5cm}

 {\bf Question: what is the exact value of $RR(C^*(L_F))$?}

{

{\small
\begin{tabular}{l} Ram\'on Flores
\\ Departmento de Geometr\'{i}a y Topolog\'{i}a, Universidad de Sevilla
\\ e-mail: \url{ramonjflores@us.es}
\\[2mm]
Sanaz Pooya
\\ Institut de Math\'{e}matiques, Universit\'{e} de Neuch\^{a}tel
\\ e-mail: \url{sanaz.pooya@unine.ch}
\\[2mm]
Alain Valette
\\ Institut de Math\'{e}matiques, Universit\'{e} de Neuch\^{a}tel
\\ e-mail: \url{alain.valette@unine.ch}
\end{tabular}}

\end{document}